\documentclass[10pt]{article}

\usepackage{definitions}
\title{On a Unified and Simplified Proof for the Ergodic Convergence Rates of PPM, PDHG and ADMM}
\author{Haihao Lu\thanks{The University of Chicago, Booth School of Business (haihao.lu@chicagobooth.edu).} \and Jinwen Yang\thanks{The University of Chicago, Department of Statistics (jinweny@uchicago.edu).}}
\date{May 2023}

\begin{document}
\maketitle

\begin{abstract}
    We present a unified viewpoint of proximal point method (PPM), primal-dual hybrid gradient (PDHG) and alternating direction method of multipliers (ADMM) for solving convex-concave primal-dual problems. This viewpoint shows the equivalence of these three algorithms upto a norm change, and it leads to a four-line simple proof of their $\mathcal O(1/k)$ ergodic rates. The simple proof technique is not limited to these three algorithms, but can also be utilized to analyze related algorithms, such as gradient descent, linearized PDHG, inexact algorithms, just to name a few.
\end{abstract}

\section{Introduction}

We study three classic algorithms, proximal point method (PPM), primal-dual hybrid gradient method (PDHG) and alternating direction method of multipliers (ADMM), for solving a convex-concave primal-dual problem,
\begin{equation}\label{eq:minimax-ppm}
    \min_{x\in\mathbb R^n}\max_{\lambda\in\mathbb R^m}\; L(x,\lambda) \ ,
\end{equation}
where $L(x,\lambda)$ is lower semicontinuous convex in the primal variable $x$, upper semicontinuous concave in the dual variable $\lambda$.  Such convex-concave primal-dual problems and their primal and/or dual forms have various applications in image processing, statistics, conic programming, etc. 

While PDHG and ADMM are extensively studied in the optimization literature and widely used in practice, the proofs of their $\mathcal O(1/k)$ ergodic convergence rate  are still viewed as ``technical'' with non-trivial mathematical manipulations. These proofs make it difficult to gain intuitions about the extrapolation steps in PDHG and the carefully-chosen cyclic rule in ADMM and to extend these algorithms in other settings. 
% The existing proofs are highly challenging, if not prohibited, to be taught in a PhD-level optimization course. 
In this paper, we provide a unified viewpoint of PPM, PDHG and ADMM. This viewpoint clearly shows that PDHG and ADMM are variants of PPM with different choices of norms. Furthermore, this viewpoint leads to a four-line simple and unified proof (Theorem \ref{thm:main}) of the $\mathcal O(1/k)$ ergodic rate for these three algorithms. Indeed, this viewpoint was also mentioned in \cite{he2012convergence,shefi2014rate}, which provides an approximate PPM interpretation of PDHG, but such connection was not further explored to obtain convergence analysis. Recently, \cite{o2020equivalence,liu2021acceleration,yan2018new} show the equivalence of ADMM/DRS and PDHG up to a change of norm, which can also be clearly obtained from the such interpretation.

We believe this perspective provides a fundamental understanding of these three algorithms, which can also facilitate developing new variants of these algorithms. Indeed, this viewpoint and the convergence result is not limited to analyze these three algorithms, and can also be easily extended to analyze other first-order methods, such as gradient descent for convex optimization, linearized PDHG, inexact algorithms, just to name a few.

% Convex-concave minimax problems are at the heart of many optimization problems arising from economics, image processing, statistics, machine learning and operation research. In numerous important cases such as Lasso, regularized logistic regression, total variation denoising etc., it is desired to minimize the sum of two functions which are possibly non-smooth but with proximal operators available under linear constraints on the variables and by dualizing the constraints, the minimax formulation arises naturally. Among the vast literature of solving such prox-friendly minimax problems, three methods and their variants are extensively studied: proximal point method (PPM), primal-dual hybrid gradient method (PDHG) and alternating direction method of multipliers (ADMM).

% In this note, we focus on three popular algorithms for solving minimax optimization problems: proximal point method (PPM), primal-dual hybrid gradient method (PDHG) and alternating direction method of multipliers (ADMM). The goal of this note is to first provide a new understanding of these three algorithms: they share update rules with formulation below
% \begin{equation*}
%     P(z^k-z^{k+1})\in\mathcal F(z^{k+1}) \ .
% \end{equation*}
% Based on that, a unified and simplified proof for their ergodic rates is presented. Seeing their proximity, an implication goes that both PDHG and ADMM can be viewed as approximations to PPM.

% \subsection{PPM, PDHG and ADMM}
In 1976, Rockafellar proposed proximal point method (PPM) in his seminal work~\cite{rockafellar1976monotone}. PPM was initially designed to solve monotone inclusion problems. As a special case of monotone inclusion, primal-dual problems can be solved utilizing PPM, as described in Algorithm \ref{alg:ppm}. PPM can be viewed as an implicit discretization of gradient flow when the objective $L(x,\lambda)$ is sufficiently differentiable. However, PPM requires solving an implicit equation jointly in the primal and the dual variable, which are often computationally challenging. A more practical class of algorithms is the operator splitting methods which alternatively update the primal and the dual variables. The two most popular operator splitting algorithms are perhaps primal-dual hybrid gradient method (PDHG)~\cite{chambolle2011first,chambolle2016ergodic} and alternating direction method of multipliers (ADMM)~\cite{glowinski1975approximation,eckstein1992douglas, boyd2011distributed}.
\begin{algorithm}
    \renewcommand{\algorithmicrequire}{\textbf{Input:}}
    \caption{PPM for \eqref{eq:minimax-ppm}}
    \label{alg:ppm}
    \begin{algorithmic}[1]
        \REQUIRE initial point $(x^0,\lambda^0)$, {step-size} $\eta >0$.
        \FOR{$k=0,1,...$}
        \STATE $(x^{k+1},\lambda^{k+1}):=\arg\min_{x}\arg\max_{\lambda}\left\{ L(x,\lambda)+\frac{1}{2\eta}\|x-x^k\|_2^2-\frac{1}{2\eta}\|\lambda-\lambda^k\|_2^2\right\}$
        \ENDFOR
    \end{algorithmic}
\end{algorithm}

PDHG was initially designed for image processing applications~\cite{chambolle2011first}, and it is recently used as the base algorithm in a new linear programming (LP) solver PDLP~\cite{applegate2021practical} that aims to significantly scale up LP. PDHG~\cite{chambolle2011first} solves the composite minimization problem
\begin{equation*}
    \min_x\; f(x)+g(Ax) \ ,
\end{equation*}
and its primal-dual formulation,
\begin{equation}\label{eq:minimax-pdhg}
    \min_x\max_{\lambda}\; L(x,\lambda)= f(x)-\lambda^TAx-g^*(\lambda) \ ,
\end{equation}
where $f$ and $g$ are lower semi-continuous convex functions, $g^*(\lambda):=\max_{y} \{\lambda^Ty-g(y)\}$ is the conjugate function of function $g$. The update rule of PDHG is presented in Algorithm \ref{alg:pdhg}. The ergodic rate of PDHG was first shown in~\cite{chambolle2011first}, and a simplified proof was presented in \cite{chambolle2016ergodic}.
\begin{algorithm}
    \renewcommand{\algorithmicrequire}{\textbf{Input:}}
    \caption{PDHG for \eqref{eq:minimax-pdhg}}
    \label{alg:pdhg}
    \begin{algorithmic}[1]
        \REQUIRE initial point $(x^0,\lambda^0)$, {step-size} $\eta >0$.
        \FOR{$k=0,1,...$}
        \STATE $x^{k+1}:=\arg\min_{x}\left\{ f(x)-\langle A^T\lambda^k,x-x^k\rangle+\frac{1}{2\eta}\|x-x^k\|_2^2 \right\}$
        \STATE $\lambda^{k+1}:=\arg\min_{\lambda}\left\{ g^*(\lambda)+\langle A(2x^{k+1}-x^k),\lambda-\lambda^k\rangle+\frac{1}{2\eta}\|\lambda-\lambda^k\|_2^2 \right\}$
        \ENDFOR
    \end{algorithmic}
\end{algorithm}

ADMM solves the following linearly constrained minimization problem:
\begin{align}\label{eq:problem-admm}
    \begin{split}
        & \ \min_{x,y}\; f(x)+g(y) \\
        & \quad \mathrm{s.t.}\; Ax+By=b \ ,
    \end{split}
\end{align} and its primal-dual form

% \begin{align}\label{eq:minimax-admm}
%     \min_x \max_{\lambda} L(x,\lambda) = f(x)-\lambda^TAx+b^T\lambda-g^*(B^T\lambda) \ .
% \end{align}
\begin{align}\label{eq:minimax-admm}
    \min_x \max_{\lambda}\; L(x,y,\lambda) = f(x)+g(y)-\lambda^T(Ax+By-b) \ ,
\end{align} 
where $f$ and $g$ are lower semi-continuous convex functions. The update rule of ADMM is presented in Algorithm \ref{alg:admm}. ADMM can be viewed as a Douglas-Rachford splitting method~\cite{glowinski1975approximation} for optimizing the dual problem to \eqref{eq:problem-admm}. The eventual convergence result of ADMM was presented in~\cite{eckstein1992douglas}, and the first $\mathcal O(1/k)$ ergodic convergence rate of ADMM was obtained in \cite{he20121}. It is used as the base algorithm for the conic programming solver SCS~\cite{o2016conic} and the quadratic programming solver OSQP~\cite{stellato2020osqp}. More applications and a comprehensive discussion on ADMM can be found in \cite{boyd2011distributed}.

\begin{algorithm}
    \renewcommand{\algorithmicrequire}{\textbf{Input:}}
    \caption{ADMM for \eqref{eq:problem-admm}}
    \label{alg:admm}
    \begin{algorithmic}[1]
        \REQUIRE initial point $(x^0,\lambda^0)$, {step-size} $\eta >0$.
        \FOR{$k=0,1,...$}
        \STATE $y^{k+1}:=\arg\min_y\left\{ g(y)-\langle B^T\lambda^k,y\rangle +\frac{\eta}{2}\|Ax^k+By-b\|_2^2 \right\}$
        \STATE $\lambda^{k+1}:=\lambda^k-\eta(Ax^k+By^{k+1}-b)$
        \STATE $x^{k+1}:=\arg\min_x \left\{ f(x) -\langle A^T\lambda^{k+1},x\rangle+\frac{\eta}{2}\|Ax+By^{k+1}-b\|_2^2  \right\}$
        \ENDFOR
    \end{algorithmic}
\end{algorithm}

\section{Main results}
In this section, we present a unified and simplified proof for ergodic rates of PPM, PDHG and ADMM. We start with presenting a generic algorithm and show that PPM, PDHG and ADMM are special cases of this generic algorithm with a different choice of norm. Then we show the $\mathcal O(1/k)$ ergodic rate of this generic algorithm. As a direct consequence, this shows the ergodic rate of the three algorithms, and show the equivalence of these algorithms upto a norm.

We denote $z=(x,\lambda)$ as the primal-dual solution pair and $\mathcal F(z)=\mathcal F(x,\lambda)=\begin{pmatrix}
    \partial_x L(x,\lambda) \\ -\partial_{\lambda}  L(x,\lambda)
\end{pmatrix}$ as the sub-gradient of the objective (more precisely, the gradient over the primal variable $x$ and the negative gradient over the dual variable $\lambda$). We consider a generic algorithm with iterate update rule:
\begin{equation}\label{eq:generic}
    P(z^k-z^{k+1})\in\mathcal F(z^{k+1}) \ ,
\end{equation}
where $P\in \mathbb R^{(m+n)\times (m+n)}$ is a positive semi-definite matrix.

Next, we show that the iterate updates of PPM, PDHG and ADMM all follow with \eqref{eq:generic} with a proper choice of matrix $P$.

\begin{lem}\label{lem:ppm}
    The update rule of PPM (Algorithm \ref{alg:ppm}) for the primal-dual problem \eqref{eq:minimax-ppm} is an instance of \eqref{eq:generic}
    with $P= \frac{1}{\eta}I 
    $.
\end{lem}
\begin{proof}
    This can be obtained directly from the first-order stationary condition of the update rule of PPM.
\end{proof}

\begin{lem}\label{lem:pdhg}
    The update rule of PDHG (Algorithm \ref{alg:pdhg}) for the primal-dual problem \eqref{eq:minimax-pdhg} is an instance of \eqref{eq:generic}
    with $P=\begin{pmatrix}
        \frac{1}{\eta}I & A^T \\ A & \frac{1}{\eta}I
    \end{pmatrix}$.
\end{lem}
\begin{proof}
    Notice that we have $\mathcal F(z)=\begin{pmatrix}
        \partial f(x) -A^T\lambda \\ \partial g^*(\lambda) +Ax
    \end{pmatrix}$ for the objective $L(x,\lambda)$ defined in \eqref{eq:minimax-pdhg}. The first-order optimality condition of update rules in Algorithm \ref{alg:pdhg} are 
    \begin{align*}
        \begin{split}
            & 0\in \partial f(x^{k+1})+\frac{1}{\eta} (x^{k+1}-x^k-\eta A^T\lambda^k)\\
            & 0 \in \partial g^*(\lambda^{k+1})+\frac{1}{\eta} (\lambda^{k+1}-\lambda^k+\eta A(2x^{k+1}-x^k)) \ .
        \end{split}
    \end{align*}
    Rearranging the above update rules, we reach at
    \begin{align*}
        \begin{split}
            & \frac{1}{\eta} (x^k-x^{k+1}) + A^T(\lambda^k-\lambda^{k+1}) \in \partial f(x^{k+1})- A^T\lambda^{k+1}\\
            & \frac{1}{\eta} (\lambda^k-\lambda^{k+1})+A(x^k-x^{k+1}) \in \partial g^*(\lambda^{k+1})+ Ax^{k+1}
        \end{split}
    \end{align*}
    that is, 
    \begin{equation*}
        P(z^k-z^{k+1}) \in\mathcal F(z^{k+1}) \ .
    \end{equation*}
\end{proof}

\begin{lem}\label{lem:admm}
    The update rule of ADMM (Algorithm \ref{alg:admm}) for the primal-dual problem \eqref{eq:minimax-admm} is an instance of \eqref{eq:generic}
    with $P=\begin{pmatrix}
        0 & 0 & 0 \\ 0 & \eta A^TA & -A^T \\ 0 & -A & \frac{1}{\eta}I
    \end{pmatrix}$.
\end{lem}
\begin{proof}
    Notice that we have $\mathcal F(z)=\begin{pmatrix}
        \partial g(y^{k+1})-B^T\lambda^{k+1} \\ \partial f(x^{k+1})-A^T\lambda^{k+1} \\
        Ax^{k+1}+By^{k+1}-b
    \end{pmatrix}$ for the objective $L(x,y,\lambda)$ defined in \eqref{eq:minimax-admm}. By utilizing the first-order optimality conditions, we can rewrite the three update rules in Algorithm \ref{alg:pdhg} on $y^{k+1}, \lambda^{k+1}, x^{k+1}$ as
    
    % First, the update of $y^{k+1}$ in Algorithm \ref{alg:admm} can be rewritten as
    \begin{align}\label{eq:up-0}
    \begin{split}
        0& \in \partial g(y^{k+1})+\eta B^T(Ax^k+By^{k+1}-b-\frac{1}{\eta}\lambda^k)= \partial g(y^{k+1})+\eta B^T\frac{1}{\eta}(\lambda^k-\lambda^{k+1}-\lambda^k)=\partial g(y^{k+1})- B^T\lambda^{k+1}\ ,
    \end{split}
\end{align}
% that is,
% \begin{equation*}
%     B^T\lambda^{k+1}\in\partial g(y^{k+1})\ .
% \end{equation*}
% Thus, it holds by the property of conjugate function that
% \begin{equation}\label{eq:y}
%     y^{k+1}\in\partial g^*(B^T\lambda^{k+1}) \ .
% \end{equation}
% Moreover, the update rule of $\lambda^{k+1}$ in Algorithm \ref{alg:admm}, we have
\begin{align}\label{eq:up-2}
    \begin{split}
        \frac{1}{\eta}(\lambda^{k}-\lambda^{k+1})-A(x^k-x^{k+1}) 
        =By^{k+1}+Ax^{k+1}-b \ ,
         %\in B\partial g^*(B^T\lambda^{k+1})+Ax^{k+1}-b\ .
    \end{split}
    \end{align}
% Furthermore, the update rule of $x^{k+1}$ in Algorithm \ref{alg:admm} can be rewritten as
\begin{align}\label{eq:up-1}
    \begin{split}
        0 & \in \partial f(x^{k+1})+\eta A^T(Ax^{k+1}+By^{k+1}-b-\frac{1}{\eta}\lambda^{k+1})  = \partial f(x^{k+1})+\eta A^T(Ax^{k+1}-Ax^k+\frac{1}{\eta}\lambda^{k}-\frac{2}{\eta}\lambda^{k+1})\ .
    \end{split}
\end{align}
% that is,
% \begin{equation}
%     \eta A^TA(x^k-x^{k+1})-A^T(\lambda^k-\lambda^{k+1}) \in \partial f(x^{k+1})-A^T\lambda^{k+1} \ .
% \end{equation}
Rearranging \eqref{eq:up-0}, \eqref{eq:up-1} and \eqref{eq:up-2}, we reach
    \begin{equation*}
        P(z^k-z^{k+1})=\begin{pmatrix}
        0 & 0 & 0 \\ 0 & \eta A^TA & -A^T \\ 0 & -A & \frac{1}{\eta}I
    \end{pmatrix}\begin{pmatrix}
        y^k-y^{k+1} \\ x^k-x^{k+1} \\ \lambda^k-\lambda^{k+1}
    \end{pmatrix} \in \begin{pmatrix}
        \partial g(y^{k+1})-B^T\lambda^{k+1} \\ \partial f(x^{k+1})-A^T\lambda^{k+1} \\
        Ax^{k+1}+By^{k+1}-b
    \end{pmatrix}=\mathcal F(z^{k+1}) \ .
    \end{equation*}
\end{proof}

The above three lemmas clearly show the connections and differences among PPM, ADMM and PDHG: they share very similar update rules (i.e., \eqref{eq:generic}) with a different choice of the matrix $P$. Indeed, these $P$ norm matrices appear in the analysis of ADMM~\cite{he20121} and PDHG~\cite{chambolle2016ergodic} without many intuitions, and our viewpoint of \eqref{eq:generic} make clear where they come from. 

The next theorem is our main result, which shows the ergodic rate of any algorithm with update rule \eqref{eq:generic} when $P$ is positive semi-definite.

\begin{thm}\label{thm:main}
    Consider an iterate update rule \eqref{eq:generic} with a positive semi-definite matrix $P$ for solving a convex-concave primal-dual problem \eqref{eq:minimax-ppm}. The iterates $\{z^k=(x^k,\lambda^k)\}_{k=1,...,\infty}$ are obtained from this iterate update rule and the initial solution $z^0=(x^0,\lambda^0)$. Denote $\bar z^k=(\bar x^k,\bar \lambda^k)=\frac{1}{k}\sum_{i=1}^k z^i$ as the average iterate. Then it holds for any $k\ge 1$ and a primal-dual solution $z=(x,\lambda)$ that
    \begin{equation*}
        L(\bar x^k,\lambda)- L(x,\bar \lambda^k)\leq \frac{\|z-z^0\|^2_{P}}{2k} \ .
    \end{equation*}
\end{thm}

\begin{proof}
Denote $w^{k+1}=P(z^k-z^{k+1})\in \mathcal F(z^{k+1})$. It then holds by convexity-concavity of $L(x,\lambda)$ that
\begin{align}\label{eq:bound-gap}
    \begin{split}
        & L(x^{k+1},\lambda)- L(x,\lambda^{k+1}) \ = L(x^{k+1},\lambda)-L(x^{k+1},\lambda^{k+1})+L(x^{k+1},\lambda^{k+1})-L(x,\lambda^{k+1}) \\
         \le & \left\langle w^{k+1}, z^{k+1}-z \right\rangle \ = \left\langle z^{k}-z^{k+1}, z^{k+1}-z \right\rangle_{P} = \frac{1}{2}\|z^{k}-z\|_{P}^2-\frac{1}{2}\|z^{k+1}-z\|_{P}^2-\frac{1}{2}\|z^{k}-z^{k+1}\|_{P}^2 \\ 
         \ \leq & \frac{1}{2}\|z^{k}-z\|_{P}^2-\frac{1}{2}\|z^{k+1}-z\|_{P}^2 \ ,
    \end{split}
\end{align}
where the last inequality follows from the fact that $\|\cdot\|_{P}$ is a semi-norm. We finish the proof by noticing
\begin{align*}
    \begin{split}
        L(\bar x^k,\lambda)-L(x,\bar \lambda^k)\leq \frac 1k \sum_{i=0}^{k-1}L(x^{i+1},\lambda)-L(x,\lambda^{i+1}) \leq \frac{1}{2 k}\|z^{0}-z\|_{P}^2 \ ,
    \end{split}
\end{align*}
where the first inequality comes from convexity-concavity of $L(x,\lambda)$ and the telescoping using \eqref{eq:bound-gap}.
\end{proof}
% \begin{rem}
%     We comment that the above analysis does not require $P$ being symmetric, but just $P+P^T$ to be positive semi-definite, so that $(z^k-z^{k+1})^T P (z^k-z^{k+1}) \ge 0$. This can help design new primal-dual algorithms that lead to $\mathcal O(1/k)$ ergodic rate. A generic approach to design new algorithm is to choose $P$ such that \eqref{eq:generic} is easy to solve and $P+P^T$ is positive semi-definite.
% \end{rem}

Utilizing Theorem \ref{thm:main} and noticing that PPM (Algorithm \ref{alg:ppm}), PDHG (Algorithm \ref{alg:pdhg}) and ADMM (Algorithm \ref{alg:admm}) are special case of the update rule \eqref{eq:generic} with a different choice $P$, we directly obtain the following ergodic rate of the three algorithms.

\begin{cor}
    Consider PPM (Algorithm \ref{alg:ppm}) for solving a convex-concave primal-dual problem \eqref{eq:minimax-ppm}. Denote $\bar z^k=(\bar x^k,\bar \lambda^k)=\frac{1}{k}\sum_{i=1}^k z^i$ as the average iterate. Then it holds for any $\eta>0$, $k\ge 1$, and $z=(x,\lambda)$ that
    \begin{equation*}
        L(\bar x^k,\lambda)-L(x,\bar \lambda^k)\leq \frac{\|z-z^0\|^2_{2}}{2\eta k} \ .
    \end{equation*}
\end{cor}

\begin{cor}
    Consider PDHG (Algorithm \ref{alg:pdhg}) for solving a convex-concave primal-dual problem of the form \eqref{eq:minimax-pdhg}. Denote $\bar z^k=(\bar x^k,\bar \lambda^k)=\frac{1}{k}\sum_{i=1}^k z^i$ as the average iterate. Then it holds for any $0<\eta\le 1/\|A\|_2$, $k\ge 1$, and $z=(x,\lambda)$ that
    \begin{equation*}
        L(\bar x^k,\lambda)-L(x,\bar \lambda^k)\leq \frac{\|z-z^0\|^2_{P}}{2k} \ ,
    \end{equation*}
        where $P=\begin{pmatrix}
        \frac{1}{\eta}I & A^T \\ A & \frac{1}{\eta}I
    \end{pmatrix}$.
\end{cor}

\begin{cor}\label{cor:admm}
    Consider ADMM (Algorithm \ref{alg:admm}) for solving a convex-concave primal-dual problem of the form \eqref{eq:minimax-admm}. Denote $\bar z^k=(\bar y^k,\bar x^k,\bar \lambda^k)=\frac{1}{k}\sum_{i=1}^k z^i$ as the average iterate. Then it holds for any $\eta>0$, $k\ge 1$, and $z=(y,x,\lambda)$ that
    \begin{equation*}
        L(\bar x^k,\bar y^k,\lambda)-L(x,y,\bar \lambda^k)\leq \frac{\|z-z^0\|^2_{P}}{2k} \ ,
    \end{equation*}
        where $P=\begin{pmatrix}
        0 & 0 & 0 \\ 0 & \eta A^TA & -A^T \\ 0 & -A & \frac{1}{\eta}I
    \end{pmatrix}$. By choosing $x=x^*$ and $y=y^*$ as the optimal solutions to the primal problem and noticing the optimal primal objective value $F^*=f(x^*)+g(y^*)= L(x^*, y^*, \bar \lambda^k)$, it holds for any $\lambda$ that 
    \begin{equation}\label{eq:primal-bound}
    f(\bar x^k)+g(\bar y^k)-\lambda^T(A\bar x^k+B\bar y^k-b) -F^*\leq \frac{\|(y^*,x^*, \lambda)-(y^0,x^0,\lambda^0)\|_{P}^2}{2k} \ .
\end{equation}
\end{cor}

\begin{rem}
    The left-hand-side of \eqref{eq:primal-bound} measures the optimality and feasibility of the solution $(\bar x^k, \bar y^k)$ to the primal problem~\eqref{eq:problem-admm}, see \cite[Theorem 3.60]{beck2017first}.
    % Note that $f(\bar x^k)+g(\bar y^k)-F^*$ can be negative as $(\bar x^k,\bar z^k)$ are not guaranteed to be feasible, Corollary \ref{cor:admm} does not immediately give rate on feasibility residual $\|A\bar x^k+B\bar y^k-b\|$. However, by referring to~\cite[Theorem 3.60]{beck2017first}, we can still show that the feasibility residual exhibits sublinear rate. 
\end{rem}

\section{Extensions}
The update rules of PPM, PDHG and ADMM require solving a proximal (implicit) oracle, because the right-hand-side of \eqref{eq:generic} depends on $z^{k+1}$. This step in general can be computationally expensive. In this section, we discuss two types of extensions to avoid this issue: linearized methods and inexact methods. 
% These ideas were explored in recent works of PDHG and ADMM~\red{Add some citations}, and our viewpoint provides a unified and simplified analysis.

% A crucial ingredient of PPM and its practical variants like PDHG and ADMM is the availability of proximal oracles to the functions. There are two extensions to overcome this issue, namely linearized methods and inexact approach. In subsequent sections, we provide unified and simplified analyses respectively.

\subsection{Linearized methods}
Instead of using $F(z^{k+1})$ in ~\eqref{eq:generic}, we consider a generic update rule:
\begin{equation}\label{eq:update-linear}
    P(z^k-z^{k+1})\in F^{k+1} \ ,
\end{equation}
where $P\in \mathbb R^{(m+n)\times (m+n)}$ is a positive semi-definite matrix and $F^{k+1}$ is an approximation to the sub-differential $\mathcal F(z^{k+1})$ so that \eqref{eq:update-linear} is inexpensive to solve.

Two economic examples of the linearized methods are gradient descent for minimization problems and linearized PDHG for primal-dual problems.

\begin{itemize}
    \item \textbf{Gradient descent~\cite{nesterov2003introductory}.} Consider a minimization problem $\min_x f(x)$ where $f(x)$ is a convex and $L$-smooth function (i.e., $\nabla f(x)$ is $L$-Lipschitz continuous). Gradient descent has update rule $x^{k+1}=x^k-\eta \nabla f(x^k)$, where $\eta$ is the step-size of the algorithm. It is easy to see that gradient descent is an instance of \eqref{eq:update-linear} with $z^k=x^k$ (one can view that the Lagrangian function $L(x,y)$ is independent of $y$), $P=\frac{1}{\eta}I$ and $F^{k+1}=\nabla f(x^k)$.
    \item \textbf{Linearized PDHG~\cite{chambolle2016ergodic}.} Linearized PDHG, which is formally stated in Algorithm~\ref{alg:l-pdhg}, solves \eqref{eq:minimax-pdhg} without the need of solving proximal oracles. Similar to the proof of Lemma \ref{lem:pdhg}, we can rewrite the update rule of linearized PDHG as an instance of \eqref{eq:update-linear} with $P=\frac{1}{\eta}I$ and $F^{k+1}=\begin{pmatrix}
        \nabla f(x^k)-A^T\lambda^{k+1} \\ \nabla g^*(\lambda^k)+Ax^{k+1}
    \end{pmatrix}$.
\end{itemize}

% As two examples, 

% we show that the iterate updates of gradient descent and linearized PDHG follow with \eqref{eq:update-linear} with a proper choice of matrix $P$ and $F^{k+1}$.

% \begin{lem}
%     Consider gradient descent for solving a convex $L$-smooth minimization problem $\min_x f(x)$, i.e., $x^{k+1}=x^k-\eta \nabla f(x^k)$. Then the update rule is an instance of \eqref{eq:update-linear} with $P=\frac{1}{\eta}I$ and $F^{k+1}=\nabla f(x^k)$.
% \end{lem}
% \begin{proof}
%     This can be obtained from straightforward arithmetic manipulation of the update rule of gradient descent.
% \end{proof}

\begin{algorithm}\renewcommand{\algorithmicrequire}{\textbf{Input:}}
    \caption{Linearized PDHG for \eqref{eq:minimax-pdhg}}
    \label{alg:l-pdhg}
    \begin{algorithmic}[1]
        \REQUIRE initial point $(x^0,\lambda^0)$, {step-size} $\eta >0$.
        \FOR{$k=0,1,...$}
        \STATE $x^{k+1}:= x^k-\eta(\nabla f(x^k)-A^T\lambda^k) $
        \STATE $\lambda^{k+1}:=\lambda^k - \eta (\nabla g^*(\lambda^k)+A(2x^{k+1}-x^k))$
        %$\arg\min_{\lambda}\left\{ g^*(\lambda)+\frac{1}{2\eta}\|\lambda-(\lambda^k-\eta A(2x^{k+1}-x^k))\|_2^2 \right\}$
        \ENDFOR
    \end{algorithmic}
\end{algorithm}

% \begin{lem}\label{lem:l-pdhg}
%     The update rule of Linearized PDHG (Algorithm \ref{alg:l-pdhg}) for the primal-dual problem \eqref{eq:minimax-pdhg} is an instance of \eqref{eq:update-linear}
%     with $P=\begin{pmatrix}
%         \frac{1}{\eta}I & A^T \\ A & \frac{1}{\eta}I
%     \end{pmatrix}$ and $F^{k+1}=\begin{pmatrix}
%     \nabla f(x^k)-A^T\lambda^{k+1} \\ \nabla g^*(\lambda^{k}) +Ax^{k+1}
% \end{pmatrix}$.
% \end{lem}
% \begin{proof}
%     This can be obtained directly by rearrgement of the terms.    
% \end{proof}

To guarantee the convergence property of the algorithms, $F^{k+1}$ should well approximate $\mathcal F(z^{k+1})$, which is formalized in the following assumptions:
\begin{ass}\label{ass:gap}
    Consider iterates $\{z^k=(x^k,\lambda^k)\}$ from algorithm with update rule \eqref{eq:update-linear}. 
    \begin{enumerate}[label=(\roman*)]
        \item There exists a matrix $E\succeq 0$ such that for any $z=(x,y)$,
    \begin{equation}\label{eq:approximation}
        L(x^{k+1},\lambda)-L(x,\lambda^{k+1}) \leq \langle F^{k+1}, z^{k+1}-z\rangle + \frac{1}{2}\|z^k-z^{k+1}\|_E^2 \ .
    \end{equation}
        \item It holds that $P\succeq E$.
    \end{enumerate}
\end{ass}
% \begin{rem}
    When $F^{k+1}=F(z^{k+1})$, Assumption \ref{ass:gap} \textit{(i)} holds with $E=0$. The matrix $E$, in principle, measures the difference between $F^{k+1}$ and $F(z^{k+1})$. Assumption \ref{ass:gap} \textit{(ii)} indicates that such difference (i.e. $E$) is not too large compared to the matrix $P$.
% \end{rem}

The next theorem shows that an algorithm with update rule \eqref{eq:update-linear} has $\mathcal O(1/k)$ sublinear rate under Assumption \ref{ass:gap}:
\begin{thm}\label{thm:linear}
    Consider an iterate update rule \eqref{eq:update-linear} with a positive semi-definite matrix $P$ for solving a convex-concave primal-dual problem \eqref{eq:minimax-pdhg}. The iterates $\{z^k=(x^k,\lambda^k)\}_{k=1,...,\infty}$ are obtained from this iterate update rule~\eqref{eq:update-linear}, and the initial solution $z^0=(x^0,\lambda^0)$. Denote $\bar z^k=(\bar x^k,\bar \lambda^k)=\frac{1}{k}\sum_{i=1}^k z^i$ as the average iterate. Suppose Assumption \ref{ass:gap} holds. Then it holds for any $k\ge 1$ and a primal-dual solution $z=(x,\lambda)$ that
    \begin{equation*}
        L(\bar x^k,\lambda)- L(x,\bar \lambda^k)\leq \frac{\|z-z^0\|^2_{P}}{2k} \ .
    \end{equation*}
\end{thm}

\begin{proof}
It holds by Assumption \ref{ass:gap} that
    \begin{align*}
    \begin{split}
        L(x^{k+1},\lambda)- L(x,\lambda^{k+1})& \ \leq \langle F^{k+1},z^{k+1}-z\rangle +\frac{1}{2}\|z^k-z^{k+1}\|_E^2 = \langle z^k-z^{k+1},z^{k+1}-z\rangle_P +\frac{1}{2}\|z^k-z^{k+1}\|_E^2\\
        & \ = \frac{1}{2}\|z^k-z\|_P^2-\frac{1}{2}\|z^{k+1}-z\|_P^2-\frac{1}{2}\|z^k-z^{k+1}\|_P^2+\frac{1}{2}\|z^k-z^{k+1}\|_E^2\\
        & \ \leq \frac{1}{2}\|z^k-z\|_P^2-\frac{1}{2}\|z^{k+1}-z\|_P^2 \ ,
    \end{split}
    \end{align*}
    where the last inequality follows from $P\succeq E$. We finish the proof by noticing 
    \begin{align*}
    \begin{split}
        L(\bar x^k,\lambda)-L(x,\bar \lambda^k)\leq \frac 1k \sum_{i=0}^{k-1}L(x^{i+1},\lambda)-L(x,\lambda^{i+1}) \leq \frac{1}{2 k}\|z^{0}-z\|_{P}^2 \ .
    \end{split}
\end{align*}
\end{proof}

As a direct consequence of Theorem \ref{thm:linear}, we can directly obtain the $\mathcal O(1/k)$ sublinear convergence rate of gradient descent and linearized PDHG by checking Assumption \ref{ass:gap} holds:

% Utilizing Theorem \ref{thm:linear} and noticing that gradient descent and linearized PDHG (Algorithm \ref{alg:l-pdhg}) are special case of the update rule \eqref{eq:update-linear} with a different choice $P$ and $F^{k+1}$, we can obtain the following ergodic rate of the three algorithms by checking Assumption \ref{ass:gap} and $P\succeq E$.
\begin{cor}
    Consider gradient descent for solving a convex $L$-smooth minimization problem $\min_x f(x)$. Denote $\bar x^k=\frac{1}{k}\sum_{i=1}^k x^i$ as the average iterate. Then it holds for any $0<\eta\le 1/L$, $k\ge 1$, and $x\in \mathrm{dom}(f)$ that
    \begin{equation*}
        f(\bar x^k)-f(x)\leq \frac{\|x-x^0\|^2}{2\eta  k} \ .
    \end{equation*}
\end{cor}
\begin{proof}
We just need to verify that Assumption \ref{ass:gap} holds. From the $L$-smoothness and convexity of the function $f$, we have
    \begin{align}\label{eq:smooth}
        \begin{split}
            &f(x^{k+1})-f(x)  \ = f(x^{k+1})-f(x^k)+f(x^k)-f(x) \\
             \ \leq &\langle \nabla f(x^k),x^{k+1}-x^k \rangle +\frac{L}{2}\|x^k-x^{k+1}\|_2^2+\langle \nabla f(x^k),x^{k}-x \rangle  =\langle \nabla f(x^k),x^{k+1}-x\rangle +\frac{L}{2}\|x^k-x^{k+1}\|_2^2 \ ,
        \end{split}
    \end{align}
    which finishes the proof by utilizing Theorem \ref{thm:linear}.
    % This implies that Assumption \ref{ass:gap} holds with $E=LI$ and $P=\frac{1}{\eta}I\succeq LI=E$. 
\end{proof}

\begin{cor}
    Consider Linearized PDHG (Algorithm \ref{alg:l-pdhg}) for solving a convex-concave primal-dual problem of the form \eqref{eq:minimax-pdhg}. Denote $\bar z^k=(\bar x^k,\bar \lambda^k)=\frac{1}{k}\sum_{i=1}^k z^i$ as the average iterate. Suppose both $f$ and $g^*$ are $L$-smooth convex functions. Then it holds for any $0<\eta\le 1/(L+\|A\|_2)$, $k\ge 1$, and $z=(x,\lambda)$ that
    \begin{equation*}
        L(\bar x^k,\lambda)-L(x,\bar \lambda^k)\leq \frac{\|z-z^0\|^2_{P}}{2k} \ ,
    \end{equation*}
        where $P=\begin{pmatrix}
        \frac{1}{\eta}I & A^T \\ A & \frac{1}{\eta}I
    \end{pmatrix}$.
\end{cor}
\begin{proof}
We just need to verify that Assumption \ref{ass:gap} holds. It holds from the smoothness and convexity of functions $f$ and $g^*$ that
\begin{align}\label{eq:eq-d}
\begin{split}
    L(x^{k+1},\lambda)- L(x^{k+1},\lambda^{k+1}) & =g^*(\lambda^{k+1})-g^*(\lambda)-(x^{k+1})^TA^T(\lambda-\lambda^{k+1}) \\ 
    & \leq (\nabla g^*(\lambda^{k+1})+Ax^{k+1})^T(\lambda^{k+1}-\lambda)+\frac{L}{2}\|\lambda^k-\lambda^{k+1}\|_2^2 \ ,
\end{split}
\end{align}
\begin{align}\label{eq:eq-p}
    \begin{split}
         L(x^{k+1},\lambda^{k+1})- L(x,\lambda^{k+1}) & = f(x^{k+1})-f(x) -(\lambda^{k+1})^TA(x^{k+1}-x) \\ 
        & \leq (\nabla f(x^k)-A^T\lambda^{k+1})^T(x^{k+1}-x)+\frac{L}{2}\|x^{k}-x^{k+1}\|_2^2\ .
    \end{split}
\end{align}
Combine \eqref{eq:eq-d} and \eqref{eq:eq-p} and we have
\begin{align*}
\begin{split}
    L(x^{k+1},\lambda)-L(x,\lambda^{k+1}) & = L(x^{k+1},\lambda)-L(x^{k+1},\lambda^{k+1})+ L(x^{k+1},\lambda^{k+1})-L(x,\lambda^{k+1}) \\
    & \leq \langle F^{k+1}, z^{k+1}-z\rangle+\frac{L}{2}\|z^{k}-z^{k+1}\|_2^2 \ ,
\end{split}
\end{align*}
which implies Assumption \ref{ass:gap} \textit{(i)} holds with $E=LI$. Furthermore, with the choice of stepsize $\eta \leq \frac{1}{L+\|A\|_2}$, it holds that $P-E=\begin{pmatrix}
        \pran{\frac{1}{\eta}-L}I & A^T \\ A & \pran{\frac{1}{\eta}-L}I
    \end{pmatrix}\succeq 0$. The result follows from application of Theorem \ref{thm:linear}.
\end{proof}

\subsection{Inexact methods}
Another typical way to avoid the potentially-expensive proximal steps is to update the iterates by inexactly solving~\eqref{eq:generic}. In this section, we consider inexact methods with updates,
\begin{equation}\label{eq:inexact}
    \mathcal P(z^k-z^{k+1})+\epsilon^{k+1}\in\mathcal F(z^{k+1}) \ ,
\end{equation}
where $\epsilon^{k+1}$ is the error for iteration $k+1$. The next theorem shows that as long as the total error, $\sum_{i=1}^k \|\epsilon^i\|_2$ is controllable, then the inexact algorithm obtains a reasonable convergence guarantee when the constraint set of the problem is bounded.
\begin{thm}\label{thm:inexact}
    Consider an iterate update rule \eqref{eq:inexact} with a positive semi-definite matrix $P$ for solving a convex-concave primal-dual problem with bounded region $Z$. Denote $D$ the diameter of $Z$. The iterates $\{z^k=(x^k,\lambda^k)\}_{k=1,...,\infty}$ are obtained form the iterate update rule~\eqref{eq:inexact} and the initial solution $z^0=(x^0,\lambda^0)$. Denote $\bar z^k=(\bar x^k,\bar \lambda^k)=\frac{1}{k}\sum_{i=1}^k z^i$ as the average iterate. Then it holds for any $k\ge 1$ and a primal-dual solution $z=(x,\lambda) \in Z$ that
    \begin{equation*}
        L(\bar x^k,\lambda)-L(x,\bar \lambda^k)\leq \frac{\|z-z^0\|^2_{\mathcal P}}{2k}+\frac{D\sum_{i=1}^k\|\epsilon^{i}\|_2}{k} \ .
    \end{equation*}
\end{thm}
\begin{rem}
    Utilizing Theorem \ref{thm:inexact}, we can obtain the convergence guarantees for inexact PDHG and inexact ADMM. Notice that $\epsilon^i$ is essentially the error in (sub-)gradient of the minimization sub-problem. One can utilize standard first-order methods (such as gradient descent and accelerated gradient descent) for the sub-problem with carefully chosen tolerance $\epsilon^i$ so that $\sum_{i=1}^k \|\epsilon^i\|_2$ is bounded.
\end{rem}
\begin{proof}
Denote $w^{k+1}=(w_x^{k+1},w_{\lambda}^{k+1})\in \mathcal F(z^{k+1})$. It then holds by convexity-concavity of $L(x,\lambda)$ that
\begin{align}
    \begin{split}
        & L(x^{k+1},\lambda)- L(x,\lambda^{k+1}) \ = L(x^{k+1},\lambda)-L(x^{k+1},\lambda^{k+1})+L(x^{k+1},\lambda^{k+1})-L(x,\lambda^{k+1}) \\
         \le & \left\langle w^{k+1}, z^{k+1}-z \right\rangle \ = \left\langle z^{k}-z^{k+1}, z^{k+1}-z \right\rangle_{P}+\langle \epsilon^{k+1},z^{k+1}-z\rangle \\
         = & \frac{1}{2}\|z^{k}-z\|_{P}^2-\frac{1}{2}\|z^{k+1}-z\|_{P}^2-\frac{1}{2}\|z^{k}-z^{k+1}\|_{P}^2 +\langle \epsilon^{k+1},z^{k+1}-z\rangle \\ 
         \ \leq & \frac{1}{2}\|z^{k}-z\|_{P}^2-\frac{1}{2}\|z^{k+1}-z\|_{P}^2+\langle \epsilon^{k+1},z^{k+1}-z\rangle \\
         \ \leq & \frac{1}{2}\|z^{k}-z\|_{P}^2-\frac{1}{2}\|z^{k+1}-z\|_{P}^2+D\|\epsilon^{k+1}\|_2\ ,
    \end{split}
\end{align}
where the last inequality follows from Cauchy-Schwarz inequlity that $\langle \epsilon^{k+1},z^{k+1}-z\rangle \leq \|\epsilon^{k+1}\|_2\|z^{k+1}-z\|_2$ and $\|z^k-z\|_2\leq D$ by the definition of $D$ as diameter of $Z$. We finish the proof by noticing
\begin{align*}
    \begin{split}
        L(\bar x^k,\lambda)-L(x,\bar \lambda^k)\leq \frac 1k \sum_{i=0}^{k-1}L(x^{i+1},\lambda)-L(x,\lambda^{i+1}) \leq \frac{1}{2 k}\|z^{0}-z\|_{P}^2 +\frac{D\sum_{i=1}^k\|\epsilon^{i}\|_2}{k} \ ,
    \end{split}
\end{align*}
where the first inequality comes from convexity-concavity of $L(x,\lambda)$ and the second is from the telescoping.
\end{proof}

% \begin{rem}
%     Inexact problem can be solved by proximal gradient descent
% \end{rem}

% \section{Conclusion}
% We provide a unified viewpoint of PPM, PDHG and ADMM, which leads to a simple four-line proof for their $O(1/k)$ ergodic convergence rates. This proof can also be extended to study related algorithms, such as linearized PDHG and inexact algorithms. We 

\section*{Acknowledgement}
This work was motivated by obtaining a ``teachable'' proof on the ergodic rates of PDHG and ADMM when the first author teaches UChicago BUSN 36919 in 2023. The authors would like to thank the students of BUSN 36919 for their helpful discussions during the classes.

% \section*{Declarations}
% This work is not supported by any funding and there is no conflict of interest as far as the authors know.

\bibliographystyle{amsplain}
\bibliography{ref-papers}

\providecommand{\bysame}{\leavevmode\hbox to3em{\hrulefill}\thinspace}
\providecommand{\MR}{\relax\ifhmode\unskip\space\fi MR }
\providecommand{\MRhref}[2]{%
  \href{http://www.ams.org/mathscinet-getitem?mr=#1}{#2}
}
\providecommand{\href}[2]{#2}
\begin{thebibliography}{10}

\bibitem{applegate2021practical}
David Applegate, Mateo D{\'\i}az, Oliver Hinder, Haihao Lu, Miles Lubin,
  Brendan O'Donoghue, and Warren Schudy, \emph{Practical large-scale linear
  programming using primal-dual hybrid gradient}, Advances in Neural
  Information Processing Systems \textbf{34} (2021), 20243--20257.

\bibitem{beck2017first}
Amir Beck, \emph{First-order methods in optimization}, SIAM, 2017.

\bibitem{boyd2011distributed}
Stephen Boyd, Neal Parikh, Eric Chu, Borja Peleato, Jonathan Eckstein, et~al.,
  \emph{Distributed optimization and statistical learning via the alternating
  direction method of multipliers}, Foundations and Trends{\textregistered} in
  Machine learning \textbf{3} (2011), no.~1, 1--122.

\bibitem{chambolle2011first}
Antonin Chambolle and Thomas Pock, \emph{A first-order primal-dual algorithm
  for convex problems with applications to imaging}, Journal of mathematical
  imaging and vision \textbf{40} (2011), 120--145.

\bibitem{chambolle2016ergodic}
\bysame, \emph{On the ergodic convergence rates of a first-order primal--dual
  algorithm}, Mathematical Programming \textbf{159} (2016), no.~1-2, 253--287.

\bibitem{eckstein1992douglas}
Jonathan Eckstein and Dimitri~P Bertsekas, \emph{On the douglas—rachford
  splitting method and the proximal point algorithm for maximal monotone
  operators}, Mathematical programming \textbf{55} (1992), 293--318.

\bibitem{glowinski1975approximation}
R~Glowinski and A~Marrocco, \emph{Approximation par {\'e}l{\'e}ments finis
  d’ordre un et r{\'e}solution par p{\'e}nalisation-dualit{\'e} d’une
  classe de probl{\`e}mes non lin{\'e}aires. rairo. rech. op{\'e}r}, Recherche
  Op{\'e}rationelle \textbf{9} (1975), no.~2, 41--76.

\bibitem{he2012convergence}
Bingsheng He and Xiaoming Yuan, \emph{Convergence analysis of primal-dual
  algorithms for a saddle-point problem: from contraction perspective}, SIAM
  Journal on Imaging Sciences \textbf{5} (2012), no.~1, 119--149.

\bibitem{he20121}
\bysame, \emph{On the o(1/n) convergence rate of the douglas--rachford
  alternating direction method}, SIAM Journal on Numerical Analysis \textbf{50}
  (2012), no.~2, 700--709.

\bibitem{liu2021acceleration}
Yanli Liu, Yunbei Xu, and Wotao Yin, \emph{Acceleration of primal--dual methods
  by preconditioning and simple subproblem procedures}, Journal of Scientific
  Computing \textbf{86} (2021), no.~2, 21.

\bibitem{nesterov2003introductory}
Yurii Nesterov, \emph{Introductory lectures on convex optimization: A basic
  course}, vol.~87, Springer Science \& Business Media, 2003.

\bibitem{o2020equivalence}
Daniel O’Connor and Lieven Vandenberghe, \emph{On the equivalence of the
  primal-dual hybrid gradient method and douglas--rachford splitting},
  Mathematical Programming \textbf{179} (2020), no.~1-2, 85--108.

\bibitem{o2016conic}
Brendan O’donoghue, Eric Chu, Neal Parikh, and Stephen Boyd, \emph{Conic
  optimization via operator splitting and homogeneous self-dual embedding},
  Journal of Optimization Theory and Applications \textbf{169} (2016),
  1042--1068.

\bibitem{rockafellar1976monotone}
R~Tyrrell Rockafellar, \emph{Monotone operators and the proximal point
  algorithm}, SIAM journal on control and optimization \textbf{14} (1976),
  no.~5, 877--898.

\bibitem{shefi2014rate}
Ron Shefi and Marc Teboulle, \emph{Rate of convergence analysis of
  decomposition methods based on the proximal method of multipliers for convex
  minimization}, SIAM Journal on Optimization \textbf{24} (2014), no.~1,
  269--297.

\bibitem{stellato2020osqp}
Bartolomeo Stellato, Goran Banjac, Paul Goulart, Alberto Bemporad, and Stephen
  Boyd, \emph{Osqp: An operator splitting solver for quadratic programs},
  Mathematical Programming Computation \textbf{12} (2020), no.~4, 637--672.

\bibitem{yan2018new}
Ming Yan, \emph{A new primal--dual algorithm for minimizing the sum of three
  functions with a linear operator}, Journal of Scientific Computing
  \textbf{76} (2018), 1698--1717.

\end{thebibliography}

\end{document}